\documentclass[12pt,english]{smfart}

\usepackage{smfthm}
\usepackage{amssymb}
\usepackage{mathrsfs}

\newcommand{\Qp}{\mathbf{Q}_p}
\newcommand{\Zp}{\mathbf{Z}_p}
\newcommand{\ZZ}{\mathbf{Z}}
\newcommand{\OO}{\mathcal{O}}
\newcommand{\MM}{\mathfrak{m}}
\newcommand{\Qpbar}{\overline{\mathbf{Q}}_p}

\renewcommand{\phi}{\varphi}
\renewcommand{\projlim}{\varprojlim}
\renewcommand{\geq}{\geqslant}
\renewcommand{\leq}{\leqslant} 
\newcommand{\Gal}{\operatorname{Gal}}
\newcommand{\galp}{\Gal(\Qpbar/\Qp)}
\newcommand{\dfont}{\mathrm{D}}
\newcommand{\nwach}{\mathrm{N}}
\newcommand{\dcris}{\mathrm{D}_{\operatorname{cris}}}
\newcommand{\ind}{\operatorname{ind}}
\newcommand{\Mat}{\operatorname{Mat}}
\newcommand{\M}{\operatorname{M}}
\newcommand{\GL}{\operatorname{GL}}
\newcommand{\Id}{\operatorname{Id}}
\newcommand{\dpar}[1]{(\!( #1 )\!)}
\newcommand{\dcroc}[1]{[\![ #1 ]\!]}
\newcommand{\calA}{\mathcal{A}}
\newcommand{\calR}{\mathcal{R}}
\newcommand{\scrX}{\mathscr{X}}
\newcommand{\scrU}{\mathscr{U}}
\newcommand{\Fil}{\operatorname{Fil}}
\newcommand{\Tr}{\operatorname{Tr}}
\newcommand{\rep}{\operatorname{rep}}
\newcommand{\eps}{\varepsilon}
\newcommand{\vp}{\operatorname{val}_p}
\newcommand{\smat}[1]{\left( \begin{smallmatrix} #1 \end{smallmatrix} \right)}

\author{Laurent Berger}
\address{UMPA ENS de Lyon \\
UMR 5669 du CNRS \\
Universit\'e de Lyon}
\email{laurent.berger@ens-lyon.fr}
\urladdr{www.umpa.ens-lyon.fr/\~{}lberger/}

\date{September 2010}

\title[Local constancy for the reduction of some crystalline representations]
{Local constancy for the reduction mod $p$ of $2$-dimensional crystalline representations}



\thanks{This research is partially supported by the ANR grant CETHop (calculs effectifs en th\'eorie de Hodge $p$-adique) ANR-09-JCJC-0048-01}

\begin{document}

\begin{abstract}
Irreducible crystalline representations of dimension $2$ of $\Gal(\Qpbar/\Qp)$ depend up to twist on two parameters, the weight $k$ and the trace of frobenius $a_p$. We show that the reduction modulo $p$ of such a representation is a locally constant function of $a_p$ (with an explicit radius) and a locally constant function of the weight $k$ if $a_p \neq 0$. We then give an algorithm for computing the reductions modulo $p$ of these representations. The main ingredient is Fontaine's theory of $(\varphi,\Gamma)$-modules as well as the theory of Wach modules.
\end{abstract}

\maketitle

\tableofcontents

\setlength{\baselineskip}{18pt}

\section*{Introduction}

Let $p$ be a prime number $\neq 2$ and $E$ a finite extension of $\Qp$ with ring of integers $\OO_E$ and maximal ideal $\MM_E$ and uniformizer $\pi_E$ and residue field $k_E$. If $k \geq 2$ and $a_p \in \MM_E$, let $D_{k,a_p}$ be the filtered $\phi$-module given by $D_{k, a_p} = E e_1 \oplus E e_2$ where :
\[ \begin{cases} \phi(e_1) = p^{k-1} e_2 \\
\phi(e_2) = -e_1 + a_p e_2 
\end{cases}
\text{and}\quad
\Fil^i D_{k, a_p} = \begin{cases}
D_{k, a_p} & \text{if $i \leq 0$,} \\
E e_1 & \text{if $1 \leq i \leq k-1$,} \\
0 & \text{if $i \geq k$.}
\end{cases} \]

By the theorem of Colmez-Fontaine (th\'eor\`eme A of \cite{CF}), there exists a crystalline $E$-linear representation $V_{k,a_p}$ of $\galp$ such that $\dcris(V^*_{k,a_p}) = D_{k,a_p}$ where $V^*_{k,a_p}$ is the dual of $V_{k,a_p}$. The representation $V_{k,a_p}$ is crystalline, irreducible, and its Hodge-Tate weights are $0$ and $k-1$.  Let $T$ denote a $\galp$-stable lattice of $V_{k,a_p}$ and let $\overline{V}_{k,a_p}$ be the semisimplification of $T/\pi_E T$. It is well-known that $\overline{V}_{k,a_p}$ depends only on $V_{k,a_p}$ and not on the choice of $T$. 

We should therefore be able to describe $\overline{V}_{k,a_p}$ in terms of $k$ and $a_p$ but this seems to be a difficult problem. Note that it is easy to make a list of all semisimple $2$-dimensional $k_E$-linear representations of $\galp$: they are twists of $\ind(\omega_2^r)$ (in the notation of \cite{BR2}; $\omega_2$ is the fundamental character of level $2$) for some $r\in \ZZ$ or direct sums of two characters.

If $2 \leq k \leq p$, then the theory of Fontaine-Laffaille gives us $\overline{V}_{k,a_p} = \ind(\omega_2^{k-1})$. If $k=p+1$ or $k \geq p+2$ and $v_p(a_p)  > \lfloor (k-2)/(p-1) \rfloor$, then theorem 4.1.1, remark 4.1.2 and proposition 4.1.4 of \cite{BLZ} show that $\overline{V}_{k,a_p} = \ind(\omega_2^{k-1})$. For other values of $a_p$ we can get a few additional results by using the $p$-adic Langlands correspondence (see \cite{BUG} or conjecture 1.5 of \cite{BR2}, combined with \cite{LBa}) or by computing the reduction in specific cases using congruences of modular forms (Savitt-Stein and Buzzard, see for instance \S 6.2 of \cite{BR2}). However, no general formula is known or even conjectured.

Our first result is that the function $a_p \mapsto \overline{V}_{k,a_p}$ is locally constant with an explicit radius and that if we see $k \gg \vp(a_p)$ as an element of the weight space $\projlim_n \ZZ/p^{n-1} (p-1)\ZZ$, then $k \mapsto \overline{V}_{k,a_p}$ is locally constant for (most) $a_p \neq 0$. 

Let $\alpha(k)= \sum_{n \geq 1} \lfloor k/p^{n-1}(p-1) \rfloor$ so that for example $\alpha(k-1) \leq \lfloor (k-1)p/(p-1)^2 \rfloor$.

\begin{enonce*}{Theorem A}\label{locsthm}
If $\vp(a_p-a'_p) > 2 \cdot \vp(a_p) + \alpha(k-1)$, then $\overline{V}_{k,a'_p} = \overline{V}_{k,a_p}$.
\end{enonce*}

The fact that $\overline{V}_{k,a_p}$ is a locally constant function of $a_p$ had been observed by many people (Colmez, Fontaine, Kisin, Pa\v{s}k{\=u}nas, \dots). The novelty in theorem A is the explicit radius. The proof uses the theory of Wach modules and consists in showing that if one knows the Wach module for $V_{k,a_p}$ then one can deform it to a Wach module for $V_{k,a'_p}$ if $a'_p$ is sufficiently close to $a_p$. By being careful, one can get the explicit radius of theorem A (this method is the one which is sketched in \S 10.3 of \cite{BBH}; it is also used in \cite{MD}).

\begin{enonce*}{Theorem B}\label{tricsthm}
If $a_p \neq 0$ and $a_p^2 \notin p^{\ZZ}$ and $k > 3 \cdot \vp(a_p) + \alpha(k-1)+1$, then there exists $m=m(k,a_p)$ such that $\overline{V}_{k',a_p} = \overline{V}_{k,a_p}$ if $k' \geq k$ and $k'-k \in p^{m-1}(p-1)\ZZ$.
\end{enonce*}

The main result of \cite{BUG} shows for example that if $0 < \vp(a_p) < 1$ and if $k \neq 3 \bmod{p-1}$, then  $\overline{V}_{k,a_p}$ depends only on $k \bmod{p-1}$. The proof of theorem B consists in showing that the $V_{k,a_p}$'s occur in the families of trianguline representations constructed in \S 5.1 of \cite{C08}. Since Colmez excludes from his main result those representations for which (in his notations) ``$\delta_0(p) \in p^{\ZZ}$'', we have to exclude those $a_p$'s for which $a_p^2 \in p^{\ZZ}$ but this can probably be easily overcome by the motivated reader. On the other hand, the restriction $a_p \neq 0$ is essential since the conclusion of theorem B fails for $a_p=0$. In this case, the associated weight space is quite different: see \cite{LBDP} for a construction of a $p$-adic family of representations which interpolates the $V_{k,0}$.

Our second result is an algorithm which can be programmed and which, given the data of $k$ and $a_p \bmod{\pi_E^n}$, will return $\overline{V}_{k,a_p}$ if $n$ is large enough. This algorithm is based on Fontaine's theory of $(\phi,\Gamma)$-modules (see A.3 of \cite{F90}) and its refinement for crystalline representations, the theory of Wach modules (see \cite{LB6}). In order to give the statement of the result, we give a few reminders about the theory of $(\phi,\Gamma)$-modules for $k_E$-linear representations. Let $\Gamma$ be a group isomorphic to $\Zp^\times$ via a map $\chi : \Gamma \to \Zp^\times$. The field $k_E\dpar{X}$ is endowed with a $k_E$-linear frobenius $\phi$ given by $\phi(f)(X)=f(X^p)$ and an action of $\Gamma$ given by $\gamma(f)(X)=f((1+X)^{\chi(\gamma)}-1)$. A $(\phi,\Gamma)$-module (over $k_E$) is a finite dimensional $k_E\dpar{X}$-vector space endowed with a semilinear frobenius whose matrix satisfies $\Mat(\phi) \in \GL_d(k_E\dpar{X})$ in some basis and a commuting semilinear continuous action of $\Gamma$. By a theorem of Fontaine (see A.3.4 of \cite{F90}), the category of $(\phi,\Gamma)$-modules over $k_E$ is naturally isomorphic to the category of $k_E$-linear representations of $\galp$. The group $\Gamma$ is topologically cyclic (at least if $p\neq 2$) so that a $(\phi,\Gamma)$-module is determined by two matrices $P$ and $G$, the matrices of $\phi$ and of a topological generator $\gamma$ of $\Gamma$ in some basis. In the sequel, we denote by $\rep(P,G)$ the $k_E$-linear representation associated to the $(\phi,\Gamma)$-module determined by $P$ and $G$.

If $f(X) \in \OO_E \dcroc{X}$, set $\phi(f)(X)=f((1+X)^p-1)$ so that in particular, $\phi(X)=XQ$ where $Q=\Phi_p(1+X)$, and let $\Gamma$ act on $\OO_E \dcroc{X}$ by $\eta(f)(X)=f((1+X)^{\chi(\eta)}-1)$. Recall that $\gamma$ is a fixed topological generator of $\Gamma$; we write $\gamma_1=\gamma^{p-1}$ so that $\chi(\gamma_1)$ is a topological generator of $1+p\Zp$ and if $G$ is the matrix of $\gamma$ in some basis, then the matrix of $\gamma_1$ is $G_1=G \gamma(G) \cdots \gamma^{p-2}(G)$. 

\begin{enonce*}{Definition}
Let $W_{k,a_p}(n)$ be the set of pairs of matrices $(P,G)$ with $P,G \in \M_2(\OO_E \dcroc{X} /(\pi_E^n,\phi(X)^k))$ satisfying the following conditions :
\begin{enumerate}
\item $P\phi(G)=G\gamma(P)$;
\item $G=\Id\bmod{X}$;
\item $\det(P)=Q^{k-1}$ and $\Tr(P)=a_p \bmod{X}$;
\item if $\Pi(Y)=(Y-1)(Y-\chi(\gamma_1)^{k-1})$, then $\Pi(G_1) = 0 \bmod{Q}$.
\end{enumerate}
\end{enonce*}

If $(P,G) \in W_{k,a_p}(n)$, then we denote by $\overline{P}$ and $\overline{G}$ two matrices in $\M_2(k_E\dcroc{X})$ which are equal modulo $\phi(X)^k$ to the reductions modulo $\pi_E$ of $P$ and $G$ (note that in $k_E\dcroc{X}$, we have $\phi(X)=X^p$). They then satisfy the relation $\overline{P}\phi(\overline{G})=\overline{G}\gamma(\overline{P}) \bmod{\phi(X)^k}$ and in proposition \ref{extunik} below, we prove that we can modify $\overline{G}$ modulo $X^k$ so that $\overline{P}\phi(\overline{G})=\overline{G}\gamma(\overline{P})$ and that the resulting representation $\rep(\overline{P}, \overline{G})$ does not depend on the modification. 

\begin{enonce*}{Theorem C}
If $n \geq 1$, then $W_{k,a_p}(n)$ is nonempty and there exists $n(k,a_p) \geq 1$ with the property that if $n \geq n(k,a_p)$ and if $(\overline{P}, \overline{G})$ is the image of any $(P,G) \in W_{k,a_p}(n)$, then $\rep(\overline{P}, \overline{G})^{\mathrm{ss}} = \overline{V}_{k,a_p}^*$.
\end{enonce*}

This theorem suggests the following algorithm. Choose some integer $n \geq 1$; since the set $\M_2(\OO_E \dcroc{X} /(\pi_E^n,\phi(X)^k))$ is finite, we can determine all the elements of $W_{k,a_p}(n)$ by checking for each pair of matrices $(P,G)$ whether it satisfies conditions (1), (2), (3) and (4). For each pair $(P,G) \in W_{k,a_p}(n)$, we compute $\rep(\overline{P}, \overline{G})^{\mathrm{ss}}$. If we get two different $k_E$-linear representations from $W_{k,a_p}(n)$ in this way, then we replace $n$ by $n+1$; otherwise, $n=n(k,a_p)$ and $ \overline{V}_{k,a_p}^* = \rep(\overline{P}, \overline{G})^{\mathrm{ss}}$. The theorem above ensures that the algorithm terminates and returns the correct answer. In order to implement the algorithm, we need to be able to identify $\rep(\overline{P}, \overline{G})$ given $P$ and $G$ and one way of doing so is explained in \S\ref{ident}. The proof of theorem C is a simple application of the theory of Wach modules. It would be useful to have an effective bound for $n(k,a_p)$ and theorem A is probably an ingredient in the determination of such a bound.

\section{Crystalline representations and Wach modules}
\label{pgmsec}

Let $\Gamma$ be the group of the introduction and let $\calA_E$ be the $\pi_E$-adic completion of $\OO_E\dcroc{X}[1/X]$, so that $\calA_E$ is the ring of power series $f(X)=\sum_{n \in \ZZ} a_n X^n$ with $a_n \in \OO_E$ and $a_{-n} \to 0$ as $n \to + \infty$. The ring $\calA_E$ is endowed with an $\OO_E$-linear frobenius $\phi$ given by $\phi(f)(X)=f((1+X)^p-1)$ and an action of $\Gamma$ given by $\eta(f)(X)=f((1+X)^{\chi(\eta)}-1)$ for $\eta \in \Gamma$. An \'etale $(\phi,\Gamma)$-module (over $\OO_E$) is a finite type $\calA_E$-module $\dfont$ endowed with a semilinear frobenius such that $\phi(\dfont)$ generates $\dfont$ as an $\calA_E$-module, and a commuting semilinear continuous action of $\Gamma$. By a theorem of Fontaine (see A.3.4 of \cite{F90}), the category of \'etale $(\phi,\Gamma)$-modules over $\OO_E$ is naturally isomorphic to the category of $\OO_E$-linear representations of $\galp$ and we denote the corresponding functor by $\dfont \mapsto V(\dfont)$, and the inverse functor by $V\mapsto \dfont(V)$. If we restrict this equivalence of categories to objects killed by $\pi_E$, then we recover the equivalence described in the introduction. 

An effective Wach module of height $h$ is a free $\OO_E\dcroc{X}$-module $\nwach$ of finite rank, with a frobenius $\phi$ and an action of $\Gamma$ such that~:
\begin{enumerate}
\item $\calA_E \otimes_{\OO_E\dcroc{X}} \nwach$ is an \'etale $(\phi,\Gamma)$-module; 
\item $\Gamma$ acts trivially on $\nwach / X \nwach$;
\item $\nwach / \phi^*(\nwach)$ is killed by $Q^h$. 
\end{enumerate}

If $\nwach$ is a Wach module, then we can associate to it the $E$-linear representation $V(\nwach)=E \otimes_{\OO_E}ÊV(\calA_E \otimes_{\OO_E\dcroc{X}} \nwach)$. We can also define a filtration on $\nwach$ by $\Fil^j \nwach = \{ y \in \nwach$ such that $\phi(y) \in Q^j \cdot \nwach\}$ and the $E$-vector space $E \otimes_{\OO_E} \nwach / X \nwach$ then has the structure of a filtered $\phi$-module. By combining proposition III.4.2 and theorem III.4.4 of \cite{LB6}, we get the following result. 

\begin{prop}\label{wachrec}
If $\nwach$ is an effective Wach module of height $h$, then $V(\nwach)$ is crystalline with Hodge-Tate weights in $[-h;0]$ and $\dcris(V(\nwach)) \simeq E \otimes_{\OO_E} \nwach / X \nwach$. 

All crystalline representations with Hodge-Tate weights in $[-h;0]$ arise in this way. 
\end{prop}

The matrix of $\phi$ gives a well-defined equivalence class in $\M_d(E \otimes_{\OO_E} \OO_E\dcroc{X})$ and we have the following result, which follows from \S III.3 of \cite{LB6}.

\begin{prop}\label{wtfil}
If $\nwach$ is an effective Wach module, then the elementary divisors in the ring $E \otimes_{\OO_E} \OO_E\dcroc{X}$ of the matrix of $\phi$ are the ideals generated by $Q^{h_1},\hdots,Q^{h_d}$ where $h_1,\hdots,h_d$ are the opposites of the Hodge-Tate weights of $V(\nwach)$.
\end{prop}

Recall that $\OO_E\dcroc{X} / Q \simeq \Zp[\zeta_p]$. The $\Qp(\zeta_p)$-vector space $E \otimes_{\OO_E}\nwach / Q \nwach$ is endowed with an action of $\Gamma$ and by propositions III.2.1 and III.2.2 of \cite{LB6}, we have the following.

\begin{prop}\label{wawe}
If $\nwach$ is an effective Wach module, if $V(\nwach)$ is the associated representation viewed as a $\Qp$-linear representation, and if $\eta \in \Gamma$ is such that $\chi(\eta) \in 1+p\Zp$, then there exists a basis of $\Qp \otimes_{\Zp}\nwach / Q \nwach$ over $\Qp(\zeta_p)$ in which the matrix of $\eta$ is diagonal and its coefficients on the diagonal are the $\chi(\eta)^{h_j}$ where $h_1,\hdots,h_d$ are the opposites of the Hodge-Tate weights of $V(\nwach)$. 
\end{prop}

If $V(\nwach)$ is an $E$-linear representation with Hodge-Tate weights $h_1,\hdots,h_d$ then the Hodge-Tate weights of the underlying $\Qp$-linear representations are the $h_i$'s each counted $[E:\Qp]$ times; in particular, $\prod_{i=1}^d (\gamma_1-\chi(\gamma_1)^{h_i})=0$ on $E \otimes_{\OO_E}\nwach / Q \nwach$ where $\gamma_1=\gamma^{p-1}$.

\section{Local constancy with respect to $a_p$}\label{lcstap}

In this section, we give a proof of theorem A. The main idea is to deform a Wach module, and in order to do this we need to prove a few ``matrix modification'' results.

\begin{lemm}\label{chgbs}
If $P_0 \in \M_2(\OO_E)$ is a matrix with eigenvalues $\lambda \neq \mu$ and $\delta=\lambda-\mu$, then there exists $Y \in \M_2(\OO_E)$ such that $Y^{-1} \in \delta^{-1}  \M_2(\OO_E)$ and $Y^{-1}P_0Y = \smat{ \lambda & 0 \\ 0 & \mu}$.
\end{lemm}

\begin{proof}
The matrix $P_0$ corresponds to an endomorphism $f$ of an $E$-vector space such that $f$ preserves some lattice $M$. Let $v$ and $w$ be two eigenvectors for the eigenvalues $\lambda$ and $\mu$ such that $v$ and $w$ are in $M$ but not in $\pi_E M$. If $x \in M$, then we can write $x = \alpha v + \beta w$ so that $f(x) = \alpha \lambda v + \beta \mu w$ and solving for $\alpha v$ and $\beta w$ shows that they belong to $\delta^{-1} M$. The lemma follows by taking for $Y$ the matrix of $\{v,w\}$.
\end{proof}

\begin{coro}\label{exhz}
If $\alpha \geq 0$ and $\eps \in \OO_E$ are such that $\vp(\eps) \geq 2 \vp(\delta)+\alpha$, then there exists $H_0 \in p^\alpha \M_2(\OO_E)$ such that $\det(\Id+H_0)=1$ and $\Tr(H_0P_0)=\eps$.
\end{coro}

\begin{proof}
If $y \in \OO_E$, let $H_0 = Y \smat{y & -y \\ y & -y} Y^{-1}$ so that $\det(\Id+H_0)=1$ and $\Tr(H_0P_0)=y\delta$. If $\vp(y) \geq \vp(\delta)+\alpha$, then $H_0 \in p^\alpha \M_2(\OO_E)$ so that we can have $\Tr(H_0P_0) = \eps$ with $y \in \OO_E$ as soon as $\vp(\eps) \geq 2 \vp(\delta)+\alpha$.
\end{proof}

Let $\gamma$ be a topological generator of $\Gamma$ so that \[ \alpha(k) = \vp\left((1-\chi(\gamma))(1-\chi(\gamma)^2)\cdots(1-\chi(\gamma)^k)\right). \] The following two propositions appear already in \S 10.3 of \cite{BBH}. 

\begin{prop}\label{csth}
If $G \in \Id + X \M_d(\OO_E\dcroc{X})$ and $k \geq 2$ and $H_0 \in p^{\alpha(k-1)} \M_2(\OO_E)$, then there exists $H \in \M_d(\OO_E\dcroc{X})$ such that $H(0)=H_0$ and $H G = G \gamma(H) \bmod{X^k}$.
\end{prop}

\begin{proof}
Write $H=H_0+XH_1+\cdots+X^{k-1} H_{k-1}$ and $G=\Id+XG_1+\cdots$. We prove by induction on $r\geq 0$ that $H_r \in ((1-\chi(\gamma))\cdots(1-\chi(\gamma)^r))^{-1} p^{\alpha(k-1)} \M_d(\OO_E)$. If $r=0$ this is the hypothesis and if $r \geq 1$, then looking at the coefficient of $X^r$ in the equation $H G = G \gamma(H) \bmod{X^k}$ shows that $(1-\chi(\gamma)^r)H_r \in ((1-\chi(\gamma))\cdots(1-\chi(\gamma)^{r-1}))^{-1} p^{\alpha(k-1)} \M_d(\OO_E)$ which completes the induction.
\end{proof}

\begin{prop}\label{mmodr}
Let $G \in \Id + X \M_d(\OO_E\dcroc{X})$ and $P \in \M_d(\OO_E\dcroc{X})$ satsify $\det(P)=Q^{k-1}$ and $P\phi(G) = G \gamma(P)$. 

If $H_0 \in p^{\alpha(k-1)} \M_2(\OO_E)$, then there exists $G' \in \Id + X \M_d(\OO_E\dcroc{X})$ and $H \in  \M_d(\OO_E\dcroc{X})$ with $H(0)=H_0$ such that if $P'=(\Id+H)P$, then $P'\phi(G') = G' \gamma(P')$.
\end{prop}

\begin{proof}
Let $H$ be the matrix constructed in proposition \ref{csth}  and let $G'_k=G$ so that we have $G_k'-P'\phi(G'_k)\gamma(P')^{-1} = X^k R_k \in X^k \M_d(\OO_E\dcroc{X})$. Assume that $j \geq k$ and that we have a matrix $G'_j$ such that \[ G'_j-P'\phi(G'_j)\gamma(P')^{-1} = X^j R_j \in X^j \M_d(\OO_E \dcroc{X}). \] If $S_j \in \M_d(\OO_E)$ and if we set $G'_{j+1}=G'_j+X^j S_j$, then 
\begin{align*} 
G'_{j+1} -P'\phi(G'_{j+1} )\gamma(P')^{-1} & = G'_j-P'\phi(G'_j)\gamma(P')^{-1} + X^j S_j - P' X^j Q^j S_j \gamma(P')^{-1} \\
& = X^j( R_j +S_j - Q^{j-k+1} P' S_j Q^{k-1} \gamma(P')^{-1}),
\end{align*}
and we can find $S_j$ such that $R_j +S_j - Q^{j-k+1} P' S_j Q^{k-1} \gamma(P')^{-1} \in X \M_d(\OO_E \dcroc{X})$ since the map $S \mapsto S - p^{j-k+1} P'(0) \cdot S \cdot (Q^{k-1} \gamma(P')^{-1})(0)$ is obviously a bijection from $\M_d(\OO_E)$ to itself. By induction on $j \geq k$, this allows us to find a sequence $(G'_j)_{j \geq k}$ which converges for the $X$-adic topology to a matrix $G'$ satisfying $P'\phi(G') = G' \gamma(P')$.
\end{proof}

\begin{proof}[Proof of theorem A]
The representation $V_{k,a_p}^*$ is crystalline with Hodge-Tate weights $0$ and $-(k-1)$. By proposition \ref{wachrec}, one can attach to it an effective Wach module $\nwach_{k,a_p}$ of height $k-1$. If we choose a basis of $\nwach_{k,a_p}$ and denote by $P$ and $G$ the matrices of $\phi$ and $\gamma \in \Gamma$, then $P\phi(G) = G \gamma(P)$. In addition, $G \in \Id + X \M_d(\OO_E\dcroc{X})$, $\det(P)=Q^{k-1}$, $\det(P(0))=p^{k-1}$ and $\Tr(P(0))=a_p$. If $a'_p \in \OO_E$ satisfies $\vp(a_p-a'_p) \geq \vp(a_p^2-4p^{k-1}) + \alpha(k-1)$, then corollary \ref{exhz} applied to $\eps=a'_p-a_p$ and proposition \ref{mmodr} give us matrices $P'=(\Id+H)P$ and $G'$ which define a Wach module $\nwach'$ coming from a crystalline representation $V'$. 

The matrix of $\phi$ on $\dcris(V')$ is $P'(0)$ and has determinant $p^{k-1}$ and trace $a'_p$. Since $\Id+H$ is invertible, the matrices $P$ and $P'$ are equivalent and proposition \ref{wtfil} implies that the filtration on $\dcris(V')$ has weights $0$ and $-(k-1)$. This shows that $\nwach'=\nwach_{k,a'_p}$. If $\vp(a_p-a'_p) > \vp(a_p^2-4p^{k-1}) + \alpha(k-1)$, then the matrices $P'$ and $G'$ are congruent modulo $\pi_E$ to $P$ and $G$ so that $\overline{V}_{k,a'_p} = \overline{V}_{k,a_p}$. 

Finally, note that $\vp(a_p^2-4p^{k-1})=2 \vp(a_p)$ if $\vp(a_p) < (k-1)/2$ and that if $\vp(a_p) \geq (k-1)/2$, then the main result of \cite{BLZ} actually gives a better bound than $2 \vp(a_p) + \alpha(k-1)$.
\end{proof}


\section{Local constancy with respect to $k$}\label{lcstk}

In this section, we give a proof of theorem B. The idea is to show that $V_{k,a_p}$ is a point in one of the families of trianguline representations constructed by Colmez in \S 5.1 of \cite{C08}. We start by briefly recalling Colmez' constructions, referring the reader to Colmez' article for more details. 

Let $\calR_E$ denote the Robba ring with coefficients in $E$. If $\delta : \Qp^\times \to E^\times$ is a continuous character, then one defines a $1$-dimensional $(\phi,\Gamma)$-module $\calR(\delta)$ by $\calR(\delta) = \calR \cdot e_\delta$ where $\phi(e_\delta)=\delta(p) e_\delta$ and $\gamma(e_\delta) = \delta(\chi(\gamma)) e_\delta$. Given two characters $\delta_1$ and $\delta_2$, Colmez constructs non-trivial extensions $0 \to \calR(\delta_1) \to \dfont \to \calR(\delta_2) \to 0$ and under certain hypothesis on $\delta_1$ and $\delta_2$, the $2$-dimensional $(\phi,\Gamma)$-module $\dfont$ is \'etale in the sense of Kedlaya (see \cite{KLMT}) and therefore gives a $p$-adic representation $V(\delta_1,\delta_2)$ using Fontaine's construction. These representations are the trianguline representations of \cite{C08} (note that if $\delta_1\delta_2^{-1}$ is of the form $x^i$ with $i \geq 0$ or $|x|x^i$ with $i \geq 1$, then there are several non-isomorphic possible extensions and one needs to introduce an $\mathcal{L}$-invariant; this needs not bother us).

If $y \in E^\times$, let $\mu_y : \Qp^\times  \to E^\times$ be the character defined by $\mu_y(p)=y$ and $\mu_y |_{\Zp^\times} = 1$. Let $\chi : \Qp^\times  \to E^\times$ be the character defined by $\chi(p)=1$ and $\chi(x)=x$ if $x \in \Zp^\times$. The following result follows from the computations of \S 4.5 of \cite{C08}.

\begin{prop}\label{tricry}
If $y \in \MM_E$ and if $k-1 > \vp(y)$, then $V(\mu_y,\mu_{1/y} \chi^{1-k}) = V^*_{k,a_p}$ with $a_p = y + p^{k-1}/y$.
\end{prop}

We now recall Colmez' construction of families of trianguline representations. Recall first that there is a natural parameter space $\scrX$ for characters $\delta : \Qp^\times  \to E^\times$. Denote by $\delta(x)$ the character corresponding to a point $x \in \scrX$. 

\begin{prop}\label{trifam}
If $\delta_1$ and $\delta_2$ are two characters as above, such that $\delta_1\delta_2^{-1}(p) \notin p^{\ZZ}$, then there exists a neighborhood $\scrU$ of $(\delta_1,\delta_2) \in \scrX^2$ and a free $\OO_{\scrU}$-module $V$ of rank $2$ with an action of $\galp$ such that $V(u)=V(\delta_1(u),\delta_2(u))$ if $u \in \scrU$.
\end{prop}

\begin{proof}[Proof of theorem B]
If $k-1 > \vp(a_p)$, put $\delta_1=\mu_{a_p}$ and $\delta_2=\mu_{1/a_p} \chi^{1-k}$ so that $V(\delta_1,\delta_2) = V^*_{k,a_p+p^{k-1}/a_p}$ by proposition \ref{tricry}. Theorem A implies that if $\vp(p^{k-1}/a_p)> 2 \cdot \vp(a_p) + \alpha(k-1)$, then $\overline{V}(\delta_1,\delta_2) = \overline{V}^*_{k,a_p}$. Proposition \ref{trifam} implies the existence of a neighborhood $\scrU$ of $(\delta_1,\delta_2) \in \scrX^2$ such that $\overline{V}(\delta_1(u),\delta_2(u))$ is constant. This implies that there exists $m$ such that $\overline{V}^*_{k',a_p} = \overline{V}^*_{k,a_p}$ if $k' \geq k$ and $k'-k \in p^{m-1} (p-1) \ZZ$ and this finishes the proof of theorem B.
\end{proof}

\section{The algorithm for computing the reduction}
\label{lift}

We start by giving a proof of the main technical result which is used in order to justify that it is enough to work with truncations of $(\phi,\Gamma)$-modules. 

\begin{prop}\label{extunik}
If $1 \leq n \leq +\infty$ and $P$ and $G_k$ are two matrices in $\M_d(\OO_E/\pi_E^n \dcroc{X})$ such that $\det(P)=Q^{k-1} \times \mathrm{unit}$ and $G_k=\Id\bmod{X}$ and $P\phi(G_k)=G_k\gamma(P) \bmod{\phi(X)^k}$, then : 
\begin{enumerate}
\item there exists $G \in \M_d(\OO_E/\pi_E^n \dcroc{X})$ such that $G=G_k \bmod{X^k}$ and $P\phi(G)=G\gamma(P)$;
\item if $P'$ and $G'$ are two matrices equal to $P$ and $G$ modulo $\phi(X)^k$ and $X^k$ and satisfying the same conditions as $P$ and $G$, then $\rep(P',G')=\rep(P,G)$.
\end{enumerate}
\end{prop}

\begin{proof}
We start by proving (1). Since $\det(P) = Q^{k-1} \times \mathrm{unit}$, the same is true of $\det(\gamma(P))$ and hence we have $Q^{k-1} \gamma(P)^{-1} \in \M_d(\OO_E/\pi_E^n \dcroc{X})$. We can therefore rewrite $P\phi(G_k)=G_k\gamma(P) \bmod{\phi(X)^k}$ as \[ G_k-P\phi(G_k)\gamma(P)^{-1} \in X^k Q\M_d(\OO_E/\pi_E^n \dcroc{X}), \] since this is true after multiplying by $Q^{k-1}$ and $Q$ is not a zero divisor in $\OO_E/\pi_E^n \dcroc{X}$. Assume that $j \geq k$ and that we have a matrix $G_j$ such that \[ G_j-P\phi(G_j)\gamma(P)^{-1} = X^j R_j \in X^j \M_d(\OO_E/\pi_E^n \dcroc{X}). \] If $S_j \in \M_d(\OO_E/\pi_E^n)$ and if we set $G_{j+1}=G_j+X^j S_j$, then 
\begin{align*} 
G_{j+1} -P\phi(G_{j+1} )\gamma(P)^{-1} & = G_j-P\phi(G_j)\gamma(P)^{-1} + X^j S_j - P X^j Q^j S_j \gamma(P)^{-1} \\
& = X^j( R_j +S_j - Q^{j-k+1} P S_j Q^{k-1} \gamma(P)^{-1}),
\end{align*}
and we can find $S_j$ such that $R_j +S_j - Q^{j-k+1} P S_j Q^{k-1} \gamma(P)^{-1} \in X \M_d(\OO_E/\pi_E^n \dcroc{X})$ since the map $S \mapsto S - p^{j-k+1} P(0) \cdot S \cdot (Q^{k-1} \gamma(P)^{-1})(0)$ is obviously a bijection from $\M_d(\OO_E/\pi_E^n)$ to itself. By induction on $j \geq k$, this allows us to find a sequence $(G_j)_{j \geq k}$ which converges for the $X$-adic topology to a matrix $G$ satisfying (1).

In order to prove (2), we start by showing that there exists a matrix $M \in \GL_d(\OO_E/\pi_E^n \dcroc{X})$ such that $M^{-1} P' \phi(M)=P$. We have by hypothesis $P'=P+\phi(X)^k S$ and hence $P'=(1+X^k R) P$ with $R=SQ^k P^{-1}$. By induction and successive approximations, we only need to show that if $P'=(1+X^j R_j) P$ with $j \geq k$, then there exists $T_j \in \M_d(\OO_E/\pi_E^n)$ such that $(1+X^j T_j)^{-1} P' \phi(1+X^j T_j)=(1+X^{j+1}R_{j+1})P$. We have 
\begin{multline*} 
(1+X^j T_j)^{-1} \cdot (1+X^j R_j) \cdot P \cdot \phi(1+X^j T_j) \\ 
= (1+X^j(R_j-T_j + Q^{j-k+1}PT_j Q^{k-1}P^{-1}) + \mathrm{O}(X^{j+1})) \cdot P 
\end{multline*}
and the claim follows from the fact that $T \mapsto T - p^{j-k+1} P(0) \cdot T \cdot (Q^{k-1}P^{-1})(0)$ is obviously a bijection from $\M_d(\OO_E/\pi_E^n)$ to itself. In order to prove (2), we are therefore reduced to the case $P=P'$. If we set $H=G'G^{-1}$, then the two equations $P\phi(G)=G\gamma(P)$ and $P\phi(G')=G'\gamma(P)$ give $P \phi(H) = H P$, with $H=\Id \bmod{X^k}$. Let $H_0=H$  and set $H_{m+1} = P \phi(H_m) P^{-1}$. Since $H=\Id \bmod{X^k}$, we can write $H_0 = \Id+X^{k-1} \phi^0(X) R_0$ and an easy induction shows that we can write $H_m=\Id+X^{k-1} \phi^m(X) R_m$ with $R_m \in \M_d(\OO_E/\pi_E^n \dcroc{X})$ so that $H_m \to \Id$ as $m \to + \infty$. The equation $P \phi(H) = H P$ implies that $H_m=H$ for all $m \geq 0$ and we are done.
\end{proof}

We now give a proof of theorem C, which we recall here. 

\begin{theo}\label{main}
If $n \geq 1$, then $W_{k,a_p}(n)$ is nonempty and there exists $n(k,a_p) \geq 1$ with the property that if $n \geq n(k,a_p)$ and if $(\overline{P}, \overline{G})$ is the image of any $(P,G) \in W_{k,a_p}(n)$, then $\rep(\overline{P}, \overline{G})^{\mathrm{ss}} = \overline{V}_{k,a_p}^*$.
\end{theo}

\begin{proof}
The representation $V_{k,a_p}^*$ is a crystalline representation with Hodge-Tate weights $-(k-1)$ and $0$ so that by proposition \ref{wachrec}, there exists an effective Wach module $\nwach_{k,a_p}$ of height $k-1$ with the property that $V(\nwach_{k,a_p}) \simeq V_{k,a_p}^*$. If $P$ and $G$ are the matrices of $\phi$ and $\gamma$ in some basis of $\nwach_{k,a_p}$ then they obviously satisfy the equation $P\phi(G) = G \gamma(P)$ and $G=\Id\bmod{X}$ by definition. The determinant of $V_{k,a_p}^*$ is $\chi^{k-1}$ so that $\det(P)=Q^{k-1} \times u$ with $u \in 1+X \OO_E\dcroc{X}$. The map $v \mapsto \phi(v)/v$ from $1+X \OO_E\dcroc{X}$ to itself is a bijection and since $p \neq 2$, every element of $1+X \OO_E\dcroc{X}$ has a square root. We can therefore modify $P$ and $G$ accordingly so that $\det(P)=Q^{k-1}$. The fact that $\dcris(V_{k,a_p}^*)=D_{k,a_p} = E \otimes_{\OO_E} \nwach_{k,a_p} / X \nwach_{k,a_p}$ implies that $\Tr(P)=a_p \bmod{X}$. Finally by proposition \ref{wawe}, the operator $(\gamma_1-1)(\gamma_1-\chi(\gamma_1)^{k-1})$ is zero on $E \otimes_{\OO_E} \nwach_{k,a_p} / Q \nwach_{k,a_p}$ so that if $G_1=G \gamma(G) \cdots \gamma^{p-2}(G)$ and $\Pi(Y)=(Y-1)(Y-\chi(\gamma_1)^{(k-1)})$, then $\Pi(G_1) = 0 \bmod{Q}$. This shows that the images of $P$ and $G$ in $\M_2(\OO_E \dcroc{X} / (\pi_E^n,\phi(X)^k))$ belong to $W_{k,a_p}(n)$ for all $n \geq 1$ so that $W_{k,a_p}(n)$ is nonempty.

We now prove the existence of $n(k,a_p)$. There are only finitely many semisimple $k_E$-linear $2$-dimensional representations of $\galp$ so that if for infinitely many $n$ there exists $(P,G) \in W_{k,a_p}(n)$ whose image $(\overline{P}, \overline{G})$ satisfies $\rep(\overline{P}, \overline{G})^{\mathrm{ss}} \neq \overline{V}_{k,a_p}^*$ then there exists some semisimple $k_E$-linear $2$-dimensional representation $U$ of $\galp$ which arises from $(P,G) \in W_{k,a_p}(n)$ for infinitely many $n$'s. By a standard compacity argument (recall that the $W_{k,a_p}(n)$ are finite sets), this implies that we can find a compatible sequence $(P_n,G_n)_{n \geq 1}$ with each term ``reducing mod $\pi_E$'' to $U$. The $P_n$ and the $G_n$ converge to $P$ and $G$ in $\M_2(\OO_E \dcroc{X})$ and $P$ and $G$ still satisfy conditions (1), (2), (3) and (4) of the definition of $W_{k,a_p}(n)$ since these conditions are continuous. In particular, conditions (1), (2) and the first part of (3) imply that $P$ and $G$ define a Wach module, which then comes from a crystalline representation $V$. Condition (3) then implies that $\dcris(V) \simeq D_{k,a_p}$ as $\phi$-modules while condition (4) along with proposition \ref{wawe} implies that the Hodge-Tate weights of $V$ belong to $\{0;-(k-1)\}$. The fact that $\dcris(V) \simeq D_{k,a_p}$ implies that the sum of the weights is $-(k-1)$ so that $\dcris(V) \simeq D_{k,a_p}$ as filtered $\phi$-modules and hence $V \simeq V_{k,a_p}^*$. But then $U = \overline{V}^{\mathrm{ss}} = \overline{V}_{k,a_p}^*$ which is a contradiction. This shows the existence of $n(k,a_p)$ and finishes the proof of the theorem.
\end{proof}

\section{Identifying mod $p$ representations}
\label{ident}

If $P$ and $G$ are two matrices in $\M_2(k_E \dcroc{X})$ such that $\det(P)=Q^{k-1} \times \mathrm{unit}$ and $G= \Id \bmod{X}$ and $P\phi(G)=G\gamma(P) \bmod{\phi(X)^k}$, then by proposition \ref{extunik} there is a well-defined $k_E$-linear representation $\rep(P,G)$ associated to $P$ and $G$. In this section, we give a crude method for determining which one it is. 

Recall that if $V$ is a $k_E$-linear representation of $\galp$, then by B.1.4 of \cite{F90} there is a $k_E\dcroc{X}$-lattice $\dfont^+(V)$ inside $\dfont(V)$ which is stable under $\phi$ and the action of $\Gamma$ and such that any other such lattice $\nwach$ satisfies $\nwach \subset \dfont^+(V)$. If $M$ is the matrix of a basis of $\nwach$ in a basis of $\dfont^+(V)$ then $\det(\phi | \nwach)=\det(\phi | \dfont^+(V)) \cdot \phi(\det(M))/\det(M)$. In particular, if $\det(\phi | \nwach)$ is $Q^{k-1} \times \mathrm{unit}$ then $\det(M)$ divides $X^{k-1}$. The algorithm for determining $\rep(P,G)$ is then the following~:

\begin{enumerate}
\item make a list of all the $k_E$-linear $2$-dimensional representations of $\galp$;
\item for each of them, compute $P$ and $G$, the matrices of $\phi$ and $\gamma$ on $\dfont^+(V)$ to precision $X^{(p+1)k+k-1}$;
\item make a list of all the $M^{-1}P\phi(M)$ and $M^{-1}G\gamma(M)$ for the finitely many $M \in \M_2(k_E\dcroc{X}/X^{(p+1)k+k-1})$ such that $\det(M)$ divides $X^{k-1}$
\end{enumerate}

Step (2) is an interesting exercise in $(\phi,\Gamma)$-modules. Note also that in step (3) we need to multiply by $M^{-1}$ so that the precision drops from $X^{(p+1)k+k-1}$ to $X^{(p+1)k}=\phi(X)^k$. This procedure gives a complete list of all possible $(P,G)$ with the corresponding representation and given a pair $(P,G)$, the representation $\rep(P,G)$ can then be determined by a simple table lookup.

\providecommand{\bysame}{\leavevmode ---\ }
\providecommand{\og}{``}
\providecommand{\fg}{''}
\providecommand{\smfandname}{et}
\providecommand{\smfedsname}{\'eds.}
\providecommand{\smfedname}{\'ed.}
\providecommand{\smfmastersthesisname}{M\'emoire}
\providecommand{\smfphdthesisname}{Th\`ese}

\end{document}